\documentclass{amsart}
\usepackage{amssymb,amsmath,amsthm,newlfont,enumerate}

\theoremstyle{plain}
\newtheorem{theorem}{Theorem}[section]
\newtheorem{lemma}[theorem]{Lemma}

\theoremstyle{definition}

\theoremstyle{remark}

\newcommand{\CC}{{\mathbb C}}
\newcommand{\DD}{{\mathbb D}}

\newcommand{\TT}{{\mathbb T}}
\newcommand{\ZZ}{{\mathbb Z}}

\newcommand{\cH}{{\mathcal H}}

\renewcommand{\Re}{\operatorname{Re}}
\renewcommand{\hat}{\widehat}

\begin{document}

\title[Coefficient formula counterexample]{A counterexample concerning the coefficient formula for de Branges--Rovnyak norms}

\author[T. Ransford]{Thomas Ransford}

\address{%
D\'epartement de math\'ematiques et de statistique, Universit\'e Laval, Qu\'ebec (QC), G1V 0A6, Canada}

\email{ransford@mat.ulaval.ca}

\thanks{Research supported by  NSERC Discovery Grant RGPIN-2026-04565}



\date{10 October 2026}

\begin{abstract}
Let $\cH(b)$ be the de Branges--Rovnyak space associated
to a non-extreme point $b$ of the unit ball of $H^\infty$,
and let $\phi=b/a$, where $a$ is the Pythagorean mate of $b$.
In an earlier article, the author showed that, if $\phi\in H^2$, then
\[
f\in\cH(b) \iff \Bigl(f\in H^2 \text{~and~}\sum_{m\ge0}\Bigl|\sum_{n\ge0}\overline{\hat{\phi}(n)}\hat{f}(m+n)\Bigr|^2<\infty\Bigr),
\]
in which case we have
\[
\|f\|_{\cH(b)}^2=\sum_{m\ge0}|\hat{f}(m)|^2
+\sum_{m\ge0}\Bigl|\sum_{n\ge0}\overline{\hat{\phi}(n)}\hat{f}(m+n)\Bigr|^2.
\]
It was also shown that there exists a space $\cH(b)$ for which $\phi\in\cap_{p<2}H^p$, yet
\[
f\in\cH(b)\not\Rightarrow \sum_{m\ge0}\Bigl|\sum_{n\ge0}\overline{\hat{\phi}(n)}\hat{f}(m+n)\Bigr|^2<\infty.
\]
Here  we complete this circle of ideas by showing that there exists
a space $\cH(b)$
for which $\phi\in\cap_{p<1}H^p$, yet
\[
\Bigl(f\in H^2\text{~and~} \sum_{m\ge0}\Bigl|\sum_{n\ge0}\overline{\hat{\phi}(n)}\hat{f}(m+n)\Bigr|^2<\infty\Bigr)
\not\Rightarrow f\in\cH(b).
\]
\end{abstract}

\maketitle

\section{Introduction}

\subsection{What this article is about}
This article is a sequel to \cite{Ra26a},
which was a study of
the coefficient formula for de Branges--Rovnyak norms.
At the end of that paper, there was an important question
left outstanding, namely \cite[Problem~8.1]{Ra26a}.
The purpose of the present paper is to answer this question.

\subsection{Background on de Branges--Rovnyak spaces}
Here is a very brief summary of the principal definitions
concerning de Branges--Rovnyak spaces
on the open unit disk $\DD$. For details,
we refer to \cite[\S2]{Ra26a} and the references cited therein.

We say that a pair of functions $(b,a)$ in $H^\infty$ is a \emph{Pythagorean pair} if $a$ is outer,  $a(0)>0$, and
 $|b|^2+|a|^2=1$ a.e.\ on the unit circle $\TT$.
The function $a$ is completely determined by $b$ and is 
called its \emph{Pythagorean mate}.
A function $b$ 
occurs as the first element of a Pythagorean pair 
if and  only if it is a non-extreme point of the unit ball of $H^\infty$.

The de Branges--Rovnyak space $\cH(b)$ associated
to a Pythagorean pair $(b,a)$
can be characterized as follows. A function $f$
belongs to $\cH(b)$ if and only if $f\in H^2$ and
$T_{\overline{b}}f\in T_{\overline{a}}(H^2)$,
where $T_{\overline{b}},T_{\overline{a}}$ are the
bounded Toeplitz operators on $H^2$
with symbols $\overline{b}$ and $\overline{a}$ respectively.
In this case, there exists a unique function $g\in H^2$
such that $T_{\overline{b}}f=T_{\overline{a}}g$, and
the $\cH(b)$-norm of $f$ can be computed via the formula
\begin{equation}\label{E:TbTa}
\|f\|_{\cH(b)}^2=\|f\|_{H^2}^2+\|g\|_{H^2}^2.
\end{equation}

If $(b,a)$ is a Pythagorean pair, then the ratio $\phi:=b/a$
belongs to the Smirnov class $N^+$. 
Conversely, every function $\phi\in N^+$ has a unique representation
as $\phi=b/a$, where $(b,a)$ is a Pythagorean pair. 

\subsection{The coefficient formula}

Let $(b,a)$ be a Pythagorean pair and let $\phi:=b/a$.
If  $f$ is holomorphic on an open neighbourhood of the closed unit disk $\overline{\DD}$, then $f\in\cH(b)$ and 
\begin{equation}\label{E:formula}
\|f\|_{\cH(b)}^2=\sum_{m\ge0}|\hat{f}(m)|^2
+\sum_{m\ge0}\Bigl|\sum_{n\ge0}\overline{\hat{\phi}(n)}\hat{f}(m+n)\Bigr|^2.
\end{equation}
This formula is implicit in \cite{Sa08} and proved explicitly in \cite{CGR10}.
In particular, it holds when $f$ is a polynomial.
Since polynomials are dense in $\cH(b)$, one might
suspect that the formula extends to arbitrary functions $f\in\cH(b)$, but in fact this turns out to be false in general.
Counterexamples can be found, for example, in \cite{EFKMR16,MR18,Ra26a}.

The article \cite{Ra26a} arose as an attempt to extend
the range of validity of \eqref{E:formula} to a larger
class of functions $f$.
One of the main results in \cite{Ra26a} was the following
theorem.

\begin{theorem}[\protect{\cite[Theorem~3.1]{Ra26a}}]\label{T:H2}
Let $(b,a)$ be a Pythagorean pair, and suppose that $\phi:=b/a\in H^2$.
Then
\begin{equation}\label{E:equiv}
f\in\cH(b) \iff \Bigl(f\in H^2 \text{~and~}\sum_{m\ge0}\Bigl|\sum_{n\ge0}\overline{\hat{\phi}(n)}\hat{f}(m+n)\Bigr|^2<\infty\Bigr),\end{equation}
in which case \eqref{E:formula} holds.
\end{theorem}

The article \cite{Ra26a} also contained an example
concerning the sharpness of the hypothesis that $\phi\in H^2$
in Theorem~\ref{T:H2}.

\begin{theorem}[\protect{\cite[Corollary~6.2]{Ra26a}}]\label{T:oldexample}
There exists a Pythagorean pair $(b,a)$  such that 
$\phi:=b/a\in\cap_{p<2}H^p$,
 yet, for some $f\in\cH(b)$, the series $\sum_{n\ge0}\overline{\hat{\phi}(n)}\hat{f}(n)$ diverges.
\end{theorem}

\subsection{The open problem and its solution}
Every example
of the failure of the formula \eqref{E:formula}
constructed until now, including that in
Theorem~\ref{T:oldexample}, has been of the same type.
In each case, it is the forward implication in
the equivalence \eqref{E:equiv}
that breaks down.
It was left as an open problem \cite[Problem~8.1]{Ra26a}
whether the reverse implication in \eqref{E:equiv} always holds.
In this article, we give a negative answer.

\begin{theorem}\label{T:newexample}
There exist a Pythagorean pair $(b,a)$  such that 
$\phi:=b/a\in\cap_{p<1}H^p$,
and a function $f\in H^2$ such that
$\sum_{m\ge0}\Bigl|\sum_{n\ge0}\overline{\hat{\phi}(n)}\hat{f}(m+n)\Bigr|^2<\infty$,
yet $f\notin \cH(b)$.
 \end{theorem}

The proof of Theorem~\ref{T:newexample} follows the
route  suggested in \cite[Proposition~8.2]{Ra26a}.
It also requires a new positive result that is of independent
interest:
that if $\phi$ is a Smirnov function
whose real part is bounded away from zero, then
the forward implication in \eqref{E:equiv} always holds.
We establish this positive result in \S\ref{S:positive},
prove Theorem~\ref{T:newexample}
in \S\ref{S:counterexample}, and  make some concluding
remarks 
in \S\ref{S:conclusion}.


\section{A new positive result}\label{S:positive}

The following theorem provides a new condition
under which the formula \eqref{E:formula} holds.

\begin{theorem}\label{T:positive}
Let $(b,a)$ be a Pythagorean pair, let $\phi:=b/a$, and suppose that
$\inf_{z\in\DD}\Re\phi(z)>0$.
If $f\in\cH(b)$, then 
\[
\sum_{m\ge0}\Bigl|\sum_{n\ge0}\overline{\hat{\phi}(n)}\hat{f}(m+n)\Bigr|^2<\infty,
\]
and the formula \eqref{E:formula} holds.
\end{theorem}

\begin{proof}
Let $f\in \cH(b)$. 
Then, as mentioned in the introduction, there is a unique $g\in H^2$
such that $T_{\overline{b}}f=T_{\overline{a}}g$,
and
\[
\|f\|_{\cH(b)}^2=\|f\|_{H^2}^2+\|g\|_{H^2}^2.
\]
Since $1/\phi\in H^\infty$ and $a=b/\phi$, we have
\[
0=T_{\overline{b}}f-T_{\overline{a}}g=T_{\overline{b}}(f-T_{1/\overline{\phi}}(g)).
\]
Hence 
\[
f-T_{1/\overline{\phi}}(g)\in \ker T_{\overline{b}}=(bH^2)^\perp
\subset ((1/\phi)bH^2)^\perp=(aH^2)^\perp=\{0\},
\]
where the last equality holds because $a$ is outer.
Therefore $T_{1/\overline{\phi}}(g)=f$.

For $N\ge0$, set 
\[
q_N(z):=\sum_{k=0}^N\hat{\phi}(k)z^k.
\]
Then, for each $m\ge0$, we have
\[
\sum_{n=0}^N\overline{\hat{\phi}(n)}\hat{f}(n+m)
=\bigl\langle f,\,z^mq_N\bigr\rangle_{H^2}
=\bigl\langle T_{1/\overline{\phi}}(g),\,z^mq_N\bigr\rangle_{H^2}
=\bigl\langle g,\,z^mq_N/\phi\bigr\rangle_{H^2}.
\]
Writing $q_N/\phi=1+r_N$, where $r_N\in z^{N+1}H^2$,
we obtain
\[
\bigl\langle g,\,z^mq_N/\phi\bigr\rangle_{H^2}
=\bigl\langle g,\,z^m+z^mr_N\bigr\rangle_{H^2}
=\hat{g}(m)+\bigl\langle g,\,z^mr_N\bigr\rangle_{H^2}.
\]
To summarize, we have
\begin{equation}\label{E:summary}
\sum_{n=0}^N\overline{\hat{\phi}(n)}\hat{f}(n+m)
=\hat{g}(m)+\bigl\langle g,\,z^mr_N\bigr\rangle_{H^2}
\quad(m\ge0,\,N\ge0).
\end{equation}
Now, since $r_N\in z^{N+1}H^2$, it is orthogonal to $1$, so
\[
1+\|r_N\|_{H^2}^2=\|1+r_N\|_{H^2}^2
=\|q_N/\phi\|_{H^2}^2
=\bigl\langle q_N/\phi,\,q_N/\phi\bigr\rangle_{H^2}
=\bigl\langle |q_N|^2,\,1/|\phi|^2\bigr\rangle_{L^2(\TT)}.
\]
By hypothesis, there exists $c>0$ such that $\Re\phi\ge c$ on $\DD$, and therefore we have
$\Re(1/\phi)\ge c/|\phi|^2$ a.e.\ on $\TT$.
Hence
\begin{align*}
\bigl\langle |q_N|^2,\,1/|\phi|^2\bigr\rangle_{L^2(\TT)}
&\le (1/c)\Re\bigl\langle |q_N|^2,\,1/\phi\bigr\rangle_{L^2(\TT)}
=(1/c)\Re\bigl\langle q_N,\,q_N/\phi\bigr\rangle_{H^2}\\
&=(1/c)\Re\bigl\langle q_N,\,1+r_N\rangle_{H^2}
=(1/c)\Re\langle q_N,1\rangle=\Re\hat{\phi}(0)/c.
\end{align*}
Therefore
\[
\|r_N\|_{H^2}\le (|\hat{\phi}(0)|/c)^{1/2},
\]
and so, by Cauchy--Schwarz,
\[
\bigl|\langle g,\,z^mr_N\rangle_{H^2}\bigr|
\le \Bigl(\sum_{n\ge m+N+1}|\hat{g}(n)|^2\Bigr)^{1/2}\Bigl(|\hat{\phi}(0)|/c\Bigr)^{1/2}.
\]
The right-hand side tends to zero as $N\to\infty$.
Substituting this information into \eqref{E:summary},
we deduce that
\[
\sum_{n=0}^\infty\overline{\hat{\phi}(n)}\hat{f}(n+m)
=\hat{g}(m) \quad(m\ge0).
\]
Hence, finally, 
\[
\sum_{m\ge0}\Bigl|\sum_{n=0}^\infty\overline{\hat{\phi}(n)}\hat{f}(n+m)
\Bigr|^2=\sum_{m\ge0}|\hat{g}(m)|^2=\|g\|_{H^2}^2<\infty,
\]
and
\[
\|f\|_{\cH(b)}^2
=\|f\|_{H^2}^2+\|g\|_{H^2}^2
=\sum_{m\ge0}|\hat{f}(m)|^2
+\sum_{m\ge0}\Bigl|\sum_{n\ge0}\overline{\hat{\phi}(n)}\hat{f}(m+n)\Bigr|^2.\qedhere
\]
\end{proof}


\section{The counterexample}\label{S:counterexample}

Theorem~\ref{T:newexample} will be a consequence
of the following, more precise result.

\begin{theorem}\label{T:counterex}
Define $\phi:\DD\to\CC$  by
\begin{equation}\label{E:counterex}
\phi(z):=\sum_{j=1}^\infty \frac{2^{-j}}{1-e^{2\pi i/2^j}z}
\quad(z\in\DD).
\end{equation}
\begin{enumerate}[\normalfont(i)]
\item
The function $\phi$ is  outer and belongs to $\cap_{p<1}H^p$.
Consequently $\phi=b/a$, where $(b,a)$ is a Pythagorean pair
with $b$ outer.
\item
There exists $f\in H^2$ such that
\[
\sum_{n\ge0}\overline{\hat{\phi}(n)}\hat{f}(n+m)=0 \quad(\forall m\ge0),
\]
yet $f\notin\cH(b)$.
\end{enumerate}
\end{theorem}

We shall use the following simple lemma.

\begin{lemma}\label{L:outer}
Let $(b,a)$ be a Pythagorean pair such that $b$ is an outer function.
Then $\|f\|_{\cH(b)}>\|f\|_{H^2}$ for all $f\in\cH(b)\setminus\{0\}$.
\end{lemma}

\begin{proof}
Let $f\in\cH(b)$. Then there exists a unique $g\in H^2$
such that $T_{\overline{b}}f=T_{\overline{a}}g$,
and $\|f\|_{\cH(b)}^2=\|f\|_{H^2}^2+\|g\|_{H^2}^2$.
It follows that $\|f\|_{\cH(b)}>\|f\|_{H^2}$ unless $g=0$.
But $g=0$ implies that $T_{\overline{b}}f=T_{\overline{a}}0=0$,
so that $f\in\ker(T_{\overline{b}})=(bH^2)^\perp=\{0\}$,
where the last equality holds because $b$ is outer. 
Thus $\|f\|_{\cH(b)}>\|f\|_{H^2}$
unless $f=0$.
\end{proof}

\begin{proof}[Proof of Theorem~\ref{T:counterex}]
(i) The series defining $\phi$ converges locally uniformly on $\DD$,
so $\phi$ is holomorphic in $\DD$.
Also, we have
\[
 \Re\frac{1}{1-e^{2\pi i/2^j} z}
 =\frac{1}{2}+\frac{1-|z|^2}{2|1-e^{2\pi i/2^j} z|^2}\ge\frac{1}{2} \quad(j\ge1).
\]
Since $\sum_{j\ge1}2^{-j}=1$, this implies that
\begin{equation}\label{E:realpart}
\Re\phi(z)\ge1/2 \quad(z\in\DD).
\end{equation}

For $0<p<1$, the principal power $\phi^p$ is holomorphic and satisfies
\[
 \cos(p\pi/2)|\phi(z)|^p\le\Re\bigl(\phi(z)^p\bigr).
\]
The mean-value property consequently gives
\[
 \sup_{0<r<1}\frac1{2\pi}\int_0^{2\pi}
 |\phi(re^{it})|^p\,dt\le\frac{\Re(\phi(0)^p)}{\cos(p\pi/2)}
=\frac{1}{\cos(p\pi/2)}.
\]
Thus $\phi\in H^p$ for every $0<p<1$, and in particular $\phi\in N^+$.
Furthermore, \eqref{E:realpart} implies that $1/\phi\in H^\infty$.
The memberships $\phi\in H^p$ and $1/\phi\in H^\infty$ together
imply that $\phi$ is outer. 

Since $\phi\in N^+$,
we can write it as  $\phi=b/a$, where $(b,a)$ is a Pythagorean pair.
Then $b=a\phi$, the product of two outer
functions, so $b$ is also outer.

\bigskip

(ii) Let us write $\gamma_k:=\hat{\phi}(k)$. 
We shall construct a sequence $(\alpha_n)\in\ell^2(\ZZ^+)$
such that $\alpha_0=1$ and
\begin{equation}\label{E:alpha}
\sum_{n=m}^\infty \gamma_{n-m}\alpha_n=0 \quad(m\ge0).
\end{equation}
Assuming that such a sequence exists,
set $f(z):=\sum_{n\ge0}\overline{\alpha}_nz^n$.
Then $f\in H^2$, and \eqref{E:alpha} gives
\[
\sum_{n\ge0}\overline{\hat{\phi}(n)}\hat{f}(n+m)=0 \quad(\forall m\ge0).
\]
Further, we have $f\notin \cH(b)$. Indeed, suppose, on the contrary,
that $f\in\cH(b)$. Since $\Re\phi\ge 1/2$ on $\DD$, Theorem~\ref{T:positive} applies, and from \eqref{E:formula}
we get $\|f\|_{\cH(b)}=\|f\|_{H^2}$. As $b$ is outer, 
Lemma~\ref{L:outer} tells us that $f=0$. This contradicts the fact that $\hat{f}(0)=\overline{\alpha}_0=1$. Thus  (ii) is proved,
subject to the construction of the sequence $(\alpha_n)$.

Before carrying out this construction,
we establish an auxiliary bound.
Let us write
$\lambda_j:=\exp(2\pi i/2^j)$.
We claim that 
\begin{equation}\label{E:claim}
\Bigl|\sum_{n=N}^{N+t-1}\gamma_{n-m}\lambda_j^{-n}-t2^{-j}\lambda_j^{-m}\Bigr|\le 2^j 
\qquad(j\ge1,~t\ge0,~N\ge m\ge0).
\end{equation}
To see this, note that, expanding the summands in
\eqref{E:counterex}, we get
\[
\gamma_k=\sum_{j=1}^\infty 2^{-j}\lambda_j^k
\quad(k\ge0).
\]
It follows that, for all integers $N\ge m\ge0$ and $t\ge0$,
\[
\sum_{n=N}^{N+t-1}\gamma_{n-m}\lambda_j^{-n}
-t2^{-j}\lambda_j^{-m}
=\sum_{\substack{\ell=1\\ \ell\ne j}}^\infty2^{-\ell}\lambda_\ell^{-m}\sum_{n=N}^{N+t-1}(\lambda_\ell/\lambda_j)^n.
\]
Now, for distinct integers $j,\ell\ge1$, the angular separation
of $\lambda_j$ and $\lambda_\ell$ lies between $\pi/2^j$
and $\pi$, so
\[
|1-\lambda_\ell/\lambda_j|\ge 2\sin(\pi/2^{j+1})\ge 2^{1-j},
\]
and hence 
\[
\Bigl|\sum_{n=N}^{N+t-1}(\lambda_\ell/\lambda_j)^n\Bigr|
\le \frac{2}{|1-\lambda_\ell/\lambda_j|}\le 2^j.
\]
It follows that
\[
\Bigl|\sum_{n=N}^{N+t-1}\gamma_{n-m}\lambda_j^{-n}-t2^{-j}\lambda_j^{-m}\Bigr|
=\Bigl| \sum_{\substack{\ell=1\\ \ell\ne j}}^\infty2^{-\ell}\lambda_\ell^{-m}\sum_{n=N}^{N+t-1}(\lambda_\ell/\lambda_j)^n\Bigr|
\le \sum_{\ell=1}^\infty 2^{-\ell}2^j=2^j.
\]
This justifies the claim \eqref{E:claim}.

We now construct the sequence $(\alpha_n)_{n\ge0}$ at the same time as a strictly increasing sequence of positive integers $(N_M)_{M\ge0}$.
The construction is inductive, starting with $\alpha_0=1$
and $N_0=1$. Let $M\ge0$, and suppose that
$N_0,\dots,N_M$ and
$\alpha_0,\dots,\alpha_{N_M-1}$  have been chosen.
For $0\le m\le M$, set
\[
R_m:=\sum_{n=m}^{N_M-1}\gamma_{n-m}\alpha_n.
\]
Choose complex numbers $d_1,\dots,d_{M+1}$ to solve the system
\begin{equation}\label{E:system}
\sum_{j=1}^{M+1}2^{-j}\lambda_j^{-m}d_j=-R_m
\quad(0\le m\le M).
\end{equation}
The matrix of this system is a Vandermonde matrix in the distinct 
numbers $\lambda_j^{-1}$, with non-zero column factors $2^{-j}$,
so the solution exists and is unique. Define
\[
C_M:=\sum_{j=1}^{M+1}2^j|d_j|,
\quad
L_M:=2+\lceil 2^{2M+4}C_M^2\rceil
\quad
\text{and}\quad
N_{M+1}:=N_M+L_M,
\]
and set
\begin{equation}\label{E:alphadef}
\alpha_n:=\frac{1}{L_M}\sum_{j=1}^{M+1}d_j\lambda_j^{-n}
\quad(N_M\le n\le N_{M+1}-1).
\end{equation}
Clearly $N_{M+1}>N_M$.
Also, 
using the definitions
of $R_m$ and then $\alpha_n$, we have,
for $0\le m\le M$ and $0\le t\le L_M$, 
\begin{align*}
&\sum_{n=m}^{N_M+t-1}\gamma_{n-m}\alpha_n\\
&=R_m+\sum_{n=N_M}^{N_M+t-1}\gamma_{n-m}\alpha_n\\
&=R_m+\sum_{n=N_M}^{N_M+t-1}\gamma_{n-m}\frac{1}{L_M}\sum_{j=1}^{M+1}d_j\lambda_j^{-n}\\
&=R_m+\frac{1}{L_M}\sum_{j=1}^{M+1}d_j\sum_{n=N_M}^{N_M+t-1}\gamma_{n-m}\lambda_j^{-n}\\
&=R_m+\frac{1}{L_M}\sum_{j=1}^{M+1}d_jt2^{-j}\lambda_j^{-m}
+\frac{1}{L_M}\sum_{j=1}^{M+1}d_j\Bigl(\sum_{n=N_M}^{N_M+t-1}\gamma_{n-m}\lambda_j^{-n}-t2^{-j}\lambda_j^{-m}\Bigr)\\
&=R_m-\frac{t}{L_M}R_m+\frac{1}{L_M}\sum_{j=1}^{M+1}d_j\Bigl(\sum_{n=N_M}^{N_M+t-1}\gamma_{n-m}\lambda_j^{-n}-t2^{-j}\lambda_j^{-m}\Bigr),
\end{align*}
where, in the last line, we used \eqref{E:system}.
Also, using \eqref{E:claim}, we have
\begin{align*}
&\Bigl|\frac{1}{L_M}\sum_{j=1}^{M+1}d_j\Bigl(\sum_{n=N_M}^{N_M+t-1}\gamma_{n-m}\lambda_j^{-n}-t2^{-j}\lambda_j^{-m}\Bigr)\Bigr|\\
&\le \frac{1}{L_M}\sum_{j=1}^{M+1}|d_j|2^j=\frac{C_M}{L_M}\le \frac{C_M}{\sqrt{L_M}}\le 2^{-M-2}.
\end{align*}
Putting everything together, we deduce that,
for $0\le m\le M$ and $0\le t\le L_M$,
\begin{equation}\label{E:overall}
\Bigl|\sum_{n=m}^{N_M+t-1}\gamma_{n-m}\alpha_n\Bigr|\le (1-t/L_M)|R_m|+2^{-M-2}.
\end{equation}
In particular, taking $t=L_M$, and recalling that $N_M+L_M=N_{M+1}$, we get
\begin{equation}\label{E:overall2}
\Bigl|\sum_{n=m}^{N_{M+1}-1}\gamma_{n-m}\alpha_n\Bigr|\le 2^{-M-2}
\quad(0\le m\le M).
\end{equation}

Now fix $m\ge0$.
For every $M>m$, the inequality \eqref{E:overall2} from stage $M-1$
shows that $|R_m|\le 2^{-M-1}$. Substituting this into
\eqref{E:overall}, we deduce that
\[
\Bigl|\sum_{n=m}^{N_M+t-1}\gamma_{n-m}\alpha_n\Bigr|\le (1-t/L_M)|R_m|+2^{-M-2} <2^{-M}
\quad(0\le t\le L_M),
\]
whence
\[
\Bigl|\sum_{n=m}^{N-1}\gamma_{n-m}\alpha_n\Bigr|\le 2^{-M}
\quad(N_M\le N\le N_{M+1}).
\]
This implies that
\[
\sum_{n=m}^\infty \gamma_{n-m}\alpha_n=0.
\]
As this holds for each $m\ge0$, we have established \eqref{E:alpha}.

All that remains is to check that the sequence $(\alpha_n)$ belongs to $\ell^2(\ZZ^+)$. For this, we go back to the definition \eqref{E:alphadef} of $\alpha_n$. For $N_M\le n<N_{M+1}$, we have
\[
|\alpha_n|
=\Bigl|\frac{1}{L_M}\sum_{j=1}^{M+1}d_j\lambda_j^{-n}\Bigr|
\le \frac{1}{L_M}\sum_{j=1}^{M+1}|d_j|\le \frac{C_M}{L_M}.
\]
Hence
\begin{align*}
\sum_{n=1}^\infty|\alpha_n|^2
&=\sum_{M=0}^\infty\sum_{n=N_M}^{N_{M+1}-1}|\alpha_n|^2
\le \sum_{M=0}^\infty(N_{M+1}-N_M)C_M^2/L_M^2\\
&=\sum_{M=0}^\infty C_M^2/L_M
\le\sum_{M=0}^\infty 2^{-2M}<\infty.
\end{align*}
The proof is complete.
\end{proof}


\section{Concluding remarks}\label{S:conclusion}

\subsection{The case when $\phi\in H^p$}
We have not yet commented on the fact that 
Theorem~\ref{T:newexample} yields a counterexample
with  $\phi\in\cap_{p<1}H^p$. This turns out to be of interest
in relation to another theorem from \cite{Ra26a}.
The following result was originally conceived as an attempt
to generalize Theorem~\ref{T:H2} from $H^2$ to other Hardy spaces $H^p$
with $p<2$.

\begin{theorem}[\protect{\cite[Theorem~4.1]{Ra26a}}]
Let $(b,a)$ be a Pythagorean pair such that $\phi:=b/a\in H^p$,
where $0<p\le 2$. Let $f\in H^2$ be a function 
whose Taylor coefficients satisfy
\begin{equation}\label{E:pcond}
\sum_{m\ge1}m^{2/p-1}|\hat{f}(m)|^2<\infty.
\end{equation}
If 
\[
\sum_{m\ge0}\Bigl|\sum_{n\ge0}\overline{\hat{\phi}(n)}\hat{f}(m+n)\Bigr|^2<\infty,
\]
then $f\in\cH(b)$ and \eqref{E:formula} holds.
\end{theorem}

When \cite{Ra26a} was published, 
it was not clear whether the extra condition \eqref{E:pcond}
was really needed for the validity of the theorem.
Thanks to Theorem~\ref{T:newexample}, 
we now know that the theorem would be false without \eqref{E:pcond},
at least if $0<p<1$.

\subsection{An application}
Theorem~\ref{T:H2} was used recently in \cite{Ra26b} in the construction 
of a Pythagorean pair $(b,a)$
and a function $f$ in the corresponding de Branges--Rovnyak space $\cH(b)$
with the following pathological property:
despite the  fact that polynomials are dense in $\cH(b)$,
the function $f$ lies outside the closed linear span
of the partial sums of its own Taylor series.

\section*{AI tools declaration}
In preparing this work, the author used ChatGPT to develop and explore strategies for the proof of Theorem~\ref{T:counterex}. The final mathematical  text was written entirely by the author, who thoroughly corrected, verified, and streamlined the proposed strategy. The author accepts full responsibility for the accuracy and integrity of the final results presented in this paper.

\bibliographystyle{plain}
\bibliography{bibfile.bib}

\end{document}